\documentclass[12pt]{smfart}

\usepackage[francais]{babel}
\usepackage{smfenum}
\usepackage{smfthm}
\usepackage{amssymb}
\usepackage[applemac]{inputenc}

\newcommand{\val}{\mathrm{val}}
\newcommand{\Qp}{\mathbf{Q}_p}
\newcommand{\Qpn}{\mathbf{Q}_{p^n}}
\newcommand{\Zp}{\mathbf{Z}_p}
\newcommand{\Fp}{\mathbf{F}_p}
\newcommand{\Fpn}{\mathbf{F}_{p^n}}
\newcommand{\Cp}{\mathbf{C}_p}
\newcommand{\ZZ}{\mathbf{Z}}
\newcommand{\OO}{\mathcal{O}}
\newcommand{\MM}{\mathfrak{m}}
\newcommand{\Fpbar}{\overline{\mathbf{F}}_p}
\newcommand{\Qpbar}{\overline{\mathbf{Q}}_p}
\newcommand{\eps}{\varepsilon}
\renewcommand{\phi}{\varphi}
\renewcommand{\projlim}{\varprojlim}
\renewcommand{\geq}{\geqslant}
\renewcommand{\leq}{\leqslant} 
\newcommand{\G}{\mathrm{G}}
\newcommand{\B}{\mathrm{B}}
\newcommand{\K}{\mathrm{K}}
\newcommand{\Z}{\mathrm{Z}}
\newcommand{\gal}{\mathcal{G}_{\Qp}}
\newcommand{\galpn}{\mathcal{G}_{\Qpn}}
\newcommand{\inert}{\mathcal{I}_{\Qp}}
\newcommand{\et}{\widetilde{\mathbf{E}}}
\newcommand{\ddiese}{\mathrm{D}^{\sharp}}
\newcommand{\dfont}{\mathrm{D}}
\newcommand{\ind}{\mathrm{ind}}
\newcommand{\Sym}{\mathrm{Sym}}

\newcommand{\Mat}{\mathrm{Mat}}

\newcommand{\smat}[1]{\left( \begin{smallmatrix} #1 \end{smallmatrix} \right)}
\newcommand{\pmat}[1]{\begin{pmatrix} #1 \end{pmatrix}}
\newcommand{\dpar}[1]{(\!( #1 )\!)}
\newcommand{\dcroc}[1]{[\![ #1 ]\!]}

\author{Laurent Berger}
\address{Universit\'e de Lyon \\
UMPA ENS Lyon \\
46 all\'ee d'Italie \\
69007 Lyon \\
France}
\email{laurent.berger@umpa.ens-lyon.fr}
\urladdr{www.umpa.ens-lyon.fr/\~{}lberger/}

\date{Septembre 2008}

\title{Repr\'esentations supersinguli\`eres de 
$\mathrm{GL}_2(\mathbf{Q}_p)$ et 
$(\varphi,\Gamma)$-modules}

\alttitle{Supersingular representations of 
$\mathrm{GL}_2(\mathbf{Q}_p)$ and 
$(\varphi,\Gamma)$-modules}


\begin{document}

\begin{abstract}
L'objet de cette note est de donner une démonstration directe du fait que si l'on applique le foncteur de Colmez à une $\overline{\mathbf{F}}_p$-représentation irréductible de dimension deux de $\mathrm{Gal}(\overline{\mathbf{Q}}_p / \mathbf{Q}_p)$, alors on trouve la restriction au sous-groupe de Borel de $\mathrm{GL}_2(\mathbf{Q}_p)$ d'une représentation supersingulière.
\end{abstract}

\begin{altabstract}
The purpose of this note is to give a direct proof of the fact that if one applies Colmez' functor to a two dimensional irreducible $\overline{\mathbf{F}}_p$-representation of $\mathrm{Gal}(\overline{\mathbf{Q}}_p / \mathbf{Q}_p)$, one gets the restriction to the Borel subgroup of $\mathrm{GL}_2(\mathbf{Q}_p)$ of a supersingular representation.
\end{altabstract} 

\maketitle

\setcounter{tocdepth}{2}
\tableofcontents

\setlength{\baselineskip}{18pt}

\section*{Introduction et notations}\label{secintro}

Cette note s'inscrit dans le cadre de la correspondance de Langlands $p$-adique et est un complément à \cite{LB5}. L'objet de ce dernier article est de démontrer la compatibilité à la réduction modulo $p$ de la {\og correspondance de Langlands $p$-adique \fg} définie par Breuil, en utilisant la réalisation découverte par Colmez de cette correspondance via les $(\phi,\Gamma)$-modules. La démonstration donnée dans \cite{LB5} est directe pour les représentations galoisiennes qui sont somme de deux caractères (quand la représentation côté $\mathrm{GL}_2$ est une somme directe de deux induites paraboliques), mais utilise un chemin assez détourné dans le cas d'une représentation galoisienne irréductible (quand la représentation côté $\mathrm{GL}_2$ est une supersingulière). L'objet de cette note est de donner une démonstration directe dans ce dernier cas. Remarquons qu'un sous-produit des calculs de \cite{LB5} est le fait que les restrictions au Borel des supersingulières restent irréductibles. Depuis, Pa\v{s}k\={u}nas a donné dans \cite{P06} une démonstration directe de ce fait et dans cette note, nous utilisons de manière essentielle le résultat de Pa\v{s}k\={u}nas. Pour garder à l'introduction une longueur raisonnable, et comme cette note fait directement suite à \cite{LB5}, nous renvoyons à ce dernier article pour les notations utilisées dans l'énoncé de notre résultat principal :

\begin{enonce*}{Théorème A}
Si $r \in \{0,\hdots,p-1\}$ et si $\chi$ est un caractère de $\Qp^\times$, alors on a un isomorphisme de représentations de $\mathrm{B}_2(\Qp)$ : $(\projlim_{\psi} \ddiese(\rho(r,\chi)))^\ast \simeq \pi(r,0,\chi)$.
\end{enonce*}

Signalons que ce théorème suit aussi des constructions très générales de Colmez dans \cite{CGL} où il est redémontré mais que la démonstration de Colmez consiste à construire l'inverse du foncteur $W \mapsto \projlim_{\psi} \ddiese(W)$ et à l'appliquer à $\pi(r,0,\chi)$ ce qui est \emph{a priori} assez différent de nos calculs.

Enfin, le lecteur que cela intéresse pourra appliquer les méthodes de cette note au cas des représentations galoisiennes de dimension $2$ qui sont sommes de deux caractères et retrouver la correspondance avec les sommes de $\pi(r,\lambda,\chi)$, ce qui consiste à redémontrer la correspondance dans ce cas-là en ne passant plus par les induites paraboliques et donc en évitant l'utilisation de la projection stéréographique.

Rappelons à présent certaines des notations qui sont utilisées dans cette note. La lettre $k$ désigne une extension finie de $\Fp$ qui est le corps des coefficients de toutes les représentations que l'on considère. On note $\omega$ le caractère cyclotomique modulo $p$ et $\mu_\lambda$ le caractère non-ramifié de $\gal$ qui envoie le frobenius arithmétique sur $\lambda^{-1}$. Si $W$ est une représentation $k$-linéaire de $\gal$, on note $\dfont(W)$ le $(\phi,\Gamma)$-module sur $k\dpar{X}$ associé à $W$ par Fontaine dans \cite{F90} et $\ddiese(W)$ le $k\dcroc{X}$-réseau de $\dfont(W)$ construit par Colmez dans \cite{CPG}. On note $\G$ pour $\mathrm{GL}_2(\Qp)$ et $\B$ pour $\mathrm{B}_2(\Qp)$ et $\K$ pour $\mathrm{GL}_2(\Zp)$ et $\Z$ pour le centre de $\G$. On note $\omega$ et $\mu_\lambda$ les caractères de $\Qp^\times$ définis par $\omega(a)=a p^{-\val(a)}$ et $\mu_\lambda(a) = \lambda^{\val(a)}$. 

\section{Quelques $(\phi,\Gamma)$-modules en caractéristique $p$}
\label{pgcp}

Si $n$ est un entier $\geq 1$, alors on note $\omega_n$ le caractère fondamental de Serre de niveau $n$ qui peut être défini de la manière suivante : on choisit $\pi_n \in \Qpbar$ tel que $\pi_n^{p^n-1} = -p$ et si $g\in \inert$, alors on pose $\omega_n(g) = g(\pi_n)/\pi_n \in \Fpbar^\times$; cette définition ne dépend pas du choix de $\pi_n$ et montre que $\omega_n$ s'étend à $\galpn$. Certains auteurs prennent plutôt pour $\pi_n$ une racine de $\pi_n^{p^n-1} = p$; cela ne change pas $\omega_n |_{\inert}$ mais notre définition a l'avantage que pour $n=1$, on a $\omega_1 = \omega$ sur $\gal$ tout entier.

Afin de décrire les $(\phi,\Gamma)$-modules associés aux représentations irréductibles en caractéristique $p$, nous devons donner une construction {\og en caractéristique $p$ \fg} de $\omega_n$. Pour cela, nous utilisons le corps $\et$ (défini dans \cite{F90}) qui intervient dans la construction des $(\phi,\Gamma)$-modules. C'est un corps algébriquement clos muni d'une action de $\gal$ et qui contient $\Fp\dpar{X}$; en particulier, il existe $Y \in \et$ tel que $Y^{(p^n-1)/(p-1)} = X$. On pose $f_g(X) = \omega(g)X/g(X)$ pour $g \in \gal$; cette série ne dépend que de l'image de $g$ dans $\Gamma$. Comme $f_g(X) \in 1 + X \Fp\dcroc{X}$ l'expression $f_g^s(X)$ a bien un sens si $s \in \Zp$.

\begin{lemm}\label{yon}
Si $g \in \galpn$ alors $g(Y) = Y \omega_n^p(g) f_g^{-\frac{p-1}{p^n-1}}(X)$. 
\end{lemm}

\begin{proof}
Rappelons que l'élément $X \in \et^+ = \projlim \OO_{\Cp}$ vaut $1-\eps$ où $\eps=(\zeta_{p^j})_{j \geq 0}$ et où $\{\zeta_{p^j}\}_{j \geq 0}$ est une suite compatible. Si $j \geq 1$, choisissons $\pi_{n,j} \in \OO_{\Cp}$ tel que :
\[ \pi_{n,j}^{\frac{p^n-1}{p-1}} = \zeta_{p^j}-1 \]
Si $g \in \galpn$, alors $g(\zeta_{p^j}-1)=\omega(g) (\zeta_{p^j}-1)f_g^{-1}(\zeta_{p^j}-1)$ et donc il existe $\omega_{n,j}(g) \in \Fpn^\times$ tel que :
\[ \frac{g(\pi_{n,j})}{\pi_{n,j}} = [\omega_{n,j}(g)] f_g^{-\frac{p-1}{p^n-1}}(\zeta_{p^j}-1), \]
où $[\cdot]$ dénote le relèvement de Teichmüller. L'application qui à $g$ associe $\omega_{n,j}(g)$ est un caractère de $\galpn$ qui ne dépend pas du choix de $\pi_{n,j}$. De plus on a :
\[ \begin{cases} (\zeta_{p^{j+1}}-1)^p=(\zeta_{p^j}-1) \cdot (1+O(p)) & \text{si $j \geq 1$} \\ 
(\zeta_p-1)^{p-1} = -p \cdot (1+O(\zeta_p-1))
\end{cases} \]
ce qui fait que $\omega_{n,j+1}^p=\omega_{n,j}$ si $j \geq 1$ et $\omega_{n,1} = \omega_n$. On en déduit aussi que l'on peut choisir les $\pi_{n,j}$ de telle manière que $\pi_{n,j+1}^p / \pi_{n,j} \in 1 + \MM_{\Cp}$. Si l'on écrit $Y=(y^{(i)}) \in \projlim \OO_{\Cp}$, alors on a $y^{(i)} = \lim_{j \to +\infty} \pi_{n,i+j}^{p^j}$ puisque les $\pi_{n,j}$ sont compatibles en ce sens que $\pi_{n,j+1}^p / \pi_{n,j} \in 1 + \MM_{\Cp}$ ce qui fait que si $g \in \galpn$, alors :
\[ \frac{g(y^{(i)})}{y^{(i)}} = [\omega_{n,i}(g)] \cdot \lim_{j \to + \infty} (f_g^{-\frac{p-1}{p^n-1}}(\zeta_{p^{i+j}}-1))^{p^j} \]
et donc que l'on a bien $g(Y)=Y \omega_n^p(g) f_g^{-\frac{p-1}{p^n-1}}(X)$ dans $\et$.
\end{proof}

On se donne à présent $1 \leq h \leq p^n-2$ et on suppose qu'il n'existe pas d'entier $r$ divisant $n$ tel que $h$ est un multiple de $(p^n-1)/(p^r-1)$. Cela revient à dire que si l'on écrit $h=h_0h_1 \hdots h_{n-1}$ en base $p$, alors l'application $i \mapsto h_i$ de $\ZZ/n\ZZ$ dans $\{0,\hdots,p-1\}$ n'admet pas de période plus petite que $n$. Dans ce cas, les caractères $\omega_n^h, \omega_n^{ph}, \hdots, \omega_n^{p^{n-1} h}$ de $\inert$ sont deux-à-deux distincts et il existe une unique représentation irréductible de $\gal$ que l'on note $\ind(\omega_n^h)$, dont le déterminant est $\omega^h$ et dont la restriction à $\inert$ est $\omega_n^h \oplus \omega_n^{ph} \oplus \cdots \oplus \omega_n^{p^{n-1} h}$. Toute représentation irréductible de dimension $n$ de $\gal$ est isomorphe à $\ind(\omega_n^h) \otimes \chi$ pour un entier $1 \leq h \leq p^n-2$ et un caractère $\chi : \gal \to k^\times$. 

\begin{prop}\label{pgomegan}
Le $(\phi,\Gamma)$-module $\dfont(\ind(\omega_n^h))$ est défini sur $\Fp\dpar{X}$ et admet une base $e_0, \hdots, e_{n-1}$ dans laquelle dans laquelle $\gamma(e_j) =  f_\gamma(X) ^{hp^j(p-1)/(p^n-1)} e_j$ si $\gamma \in \Gamma$ et $\phi(e_j) = e_{j+1}$ pour $0 \leq j \leq n-2$ et $\phi(e_{n-1}) = (-1)^{n-1} X^{-h(p-1)} e_0$.
\end{prop}

\begin{proof} 
Soit $W$ la représentation associée au $(\phi,\Gamma)$-module décrit dans la proposition. Si l'on pose $f = X^h e_0 \wedge \hdots \wedge e_{n-1}$, alors on a $\phi(f)=f$ et $\gamma(f) = \omega(\gamma)^h f$ ce qui fait que le déterminant de $W$ est bien $\omega^h$ et il suffit donc de montrer que la restriction de $\Fpn \otimes_{\Fp} W$ à $\inert$ se décompose en $\omega_n^h \oplus \omega_n^{ph} \oplus \cdots \oplus \omega_n^{p^{n-1} h}$. 

Si l'on décompose $\Fpn \otimes_{\Fp} \et$ en $\prod_{k=0}^{n-1} \et$ via l'application $x \otimes y \mapsto (\sigma^k(x)y)$ où $\sigma$ est le frobenius de $\Fpn$, alors pour $(x_0,\hdots,x_{n-1}) \in \prod_{k=0}^{n-1} \et$, on a les formules :
\begin{align*}
\phi((x_0,\hdots,x_{n-1})) & =(\phi(x_{n-1}),\phi(x_0),\hdots,\phi(x_{n-2})) \\ 
g((x_0,\hdots,x_{n-1}))& =(g(x_0),\hdots,g(x_{n-1})),
\end{align*} 
si $g \in \galpn$ (mais pas si $g \in \gal$). On choisit $\alpha \in \et$ tel que $\alpha^{p^n-1} = (-1)^{n-1}$ et on pose :
\begin{align*}
v_0 & = (\alpha Y^h,0,\hdots,0) \cdot e_0 + (0, \alpha^p Y^{ph}, \hdots,0) \cdot e_1 + \cdots (0,\hdots,0,\alpha^{p^{n-1}} Y^{p^{n-1}h}) \cdot e_{n-1} \\
v_1 & = (0,\alpha Y^h,\hdots,0) \cdot e_0 + (0,0,  \alpha^p Y^{ph}, \hdots,0) \cdot e_1 + \cdots (\alpha^{p^{n-1}} Y^{p^{n-1}h},0,\hdots,0) \cdot e_{n-1} \\
 & \vdots \\
v_{n-1} &= (0,\hdots,0,\alpha Y^h) \cdot e_0 + (\alpha^p Y^{ph},0, \hdots,0) \cdot e_1 + \cdots (0,\hdots,0,\alpha^{p^{n-1}} Y^{p^{n-1}h},0) \cdot e_{n-1} 
\end{align*}
On vérifie que les vecteurs $v_0,\hdots,v_{n-1}$ forment une base de $\Fpn \otimes_{\Fp} (\et \otimes_{\Fp\dpar{X}} \dfont(W))$. Les formules qui donnent l'action de $\phi$ impliquent que $\phi(v_j)=v_j$ ce qui fait que $v_j \in \Fpn \otimes_{\Fp} W$. Enfin, les formules qui donnent l'action de $\Gamma$ et le lemme \ref{yon} impliquent que $g(v_j) = \omega_n^{hp^{1-j}} v_j$ ce qui achève la démonstration.
\end{proof}

\section{Représentations irréductibles de dimension $2$}

On fixe à présent $r \in \{0,\hdots,p-1\}$ et $\chi = \omega^s \mu_\lambda$ un caractère de $\gal$ et on pose $W = \rho(r,\chi) = (\ind(\omega_2^{r+1})) \otimes \chi$. On sait que toute $k$-représentation de dimension $2$ de $\gal$ est isomorphe à une telle représentation. Comme corollaire immédiat de la proposition \ref{pgomegan}, on trouve le résultat suivant.

\begin{prop}\label{pgmindom}
Le $(\phi,\Gamma)$-module $\dfont(W)$ admet une base $e,f$ dans laquelle :
\[ \Mat(\phi) = \begin{pmatrix} 0 & -X^{-(r+1)(p-1)} \\ 1 & 0 \end{pmatrix} \cdot \lambda
\quad\text{et}\quad
\Mat(\gamma) = \begin{pmatrix} f_\gamma(X) ^{\frac{r+1}{p+1}} & 0 \\ 0 &   f_\gamma(X) ^{\frac{p(r+1)}{p+1}} \end{pmatrix} \cdot \omega(\gamma)^s. \]
\end{prop}

\begin{coro}\label{actpsi}
L'action de l'opérateur $\psi$ sur $\dfont(W)$ est donnée par : 
\[ \psi(\alpha e + \beta f) = \psi(\beta) \lambda^{-1} e - 
\psi(X^{(r+1)(p-1)} \alpha) \lambda^{-1} f. \]
\end{coro}

\begin{proof}
On a $\psi(\alpha e + \beta f) = \psi(- \alpha  X^{(r+1)(p-1)} \lambda^{-1} \phi(f) + \beta \lambda^{-1} \phi(e))$ et le corollaire suit alors du fait que $\psi(a\phi(b))  = \psi(a)b$.
\end{proof}

\begin{prop}\label{ddindom}
Le treillis de Colmez $\ddiese(W)$ est donné par :
\[ \ddiese(W) = k\dcroc{X} \cdot e \oplus X^r k\dcroc{X} \cdot f. \] 
\end{prop}

\begin{proof}
La définition de $\ddiese$ est donnée dans le \S 2.4 de \cite{CPG}, c'est le plus grand $k\dcroc{X}$-réseau de $\dfont$ qui est stable par $\psi$ et sur lequel $\psi$ est surjectif. Par ailleurs, le lemme 1.1.2 de \cite{LB5} montre que dans notre cas, $\ddiese(W)$ est le seul $k\dcroc{X}$-réseau de $\dfont$ qui est stable par $\psi$ et sur lequel $\psi$ est surjectif et il suffit donc de vérifier que $M = k\dcroc{X} e \oplus X^r k\dcroc{X} f$ a ces deux propriétés. Le corollaire \ref{actpsi} nous donne la formule :
\[ \psi(\alpha e + \beta X^r f)= \psi(\beta X^r) \lambda^{-1} e - \psi(\alpha X^{p-1-r}) \lambda^{-1} X^r f \] 
ce qui fait que $M$ est stable par $\psi$. Si $0 \leq t \leq p-1$, alors $\psi(X^t)=(-1)^t$ et l'application $f(X) \mapsto \psi(X^t f(X))$ de $k\dcroc{X} \to k\dcroc{X}$ est donc surjective. Comme $r$ et $p-1-r$ sont compris entre $0$ et $p-1$, l'application $\psi$ est bien surjective sur $M$ et donc $\ddiese(W) = k\dcroc{X} e \oplus X^r k\dcroc{X}  f$.
\end{proof}

\section{Construction de représentations du Borel}

On conserve $W = \rho(r,\chi) = (\ind(\omega_2^{r+1})) \otimes \chi$ et on pose comme dans \cite{LB5} :
\[ \projlim_{\psi} \ddiese(W) = \text{$\{ y = (y_0,y_1,\hdots)$ avec $y_i \in \ddiese(W)$ tels que $\psi(y_{i+1}) = y_i$ pour tout $i\geq 0 \}$,} \] 
que l'on munit de l'action de $\B$ donnée par les formules suivantes :
\begin{align*}
\left( \begin{pmatrix} x & 0 \\ 0 & x \end{pmatrix} \star v \right)_i
& = (\omega^r \chi^2)^{-1}(x) v_i;  \\
\left( \begin{pmatrix} 1 &  0 \\ 0 & p^j  
\end{pmatrix} \star v \right)_i & = v_{i-j} = \psi^j(v_i); \\
\left( \begin{pmatrix} 1 &  0 \\ 0 & a  
\end{pmatrix} \star v \right)_i & = \gamma_a^{-1}(v_i), 
\text{ o\`u $\gamma_a \in \Gamma$ est tel que $\chi_{\mathrm{cycl}}(\gamma_a) = a$;}\\
\left( \begin{pmatrix} 1 &  z \\ 0 & 1  
\end{pmatrix} \star v \right)_i & = 
\psi^j((1+X)^{p^{i+j} z} v_{i+j}),\text{ pour $i+j \geq -{\rm val}(z)$.}
\end{align*}
On pose ensuite $\Omega(W) = (\projlim_{\psi} \ddiese(W))^*$ ce qui fait de $\Omega(W)$ une représentation lisse de $\B$ dont le caractère central est $\omega^r \chi^2$. Rappelons que par la proposition 1.2.2 de \cite{LB5}, la représentation $\Omega(W)$ est irréductible. Si $y = (y_0,y_1,\hdots)$, alors par la proposition \ref{ddindom} on peut écrire $y_i=\alpha_i e + \beta_i X^r f$ et on appelle $\theta \in \Omega(W)$ la forme linéaire qui à $y$ associe $\theta(y) = \alpha_0(0)$. 

\begin{lemm}\label{acbormu}
Si $\smat{ a & b \\ 0 & d} \in \B \cap \K \Z$, alors $\smat{ a & b \\ 0 & d} \star \theta = \chi(ad) \omega^r(a) \theta$.
\end{lemm}

\begin{proof}
On a :
\begin{align*} 
\left(\pmat{ a & b \\ 0 & d} \star \theta\right)(y) 
& = \theta \left( \pmat{ a^{-1} & -b a^{-1} d^{-1} \\ 0 & d^{-1}} \star y \right) \\
& = \theta \left( \pmat{ a^{-1} & 0 \\ 0 & a^{-1}} \pmat{ 1 & -b d^{-1} \\ 0 & ad^{-1}} \star y \right) \\
& = \theta (  (\omega^r \chi^2)^{-1}(a^{-1}) \gamma_{ad^{-1}}^{-1}((1+X)^{-bd^{-1}} y)) \\
& = (\omega^r \chi^2)(a) \omega^s(a^{-1}d) \theta(y) \\
&=  \omega^r(a) \chi(ad) \theta(y),
\end{align*}
puisque $\mu_\lambda(a) = \mu_\lambda(d)$ comme $\smat{ a & b \\ 0 & d} \in \B \cap \K \Z$ ce qui fait que $\chi(a) = \chi(d) \omega^s(ad^{-1})$.
\end{proof}

Si $V$ est une représentation de $\B \cap \K \Z$, on note $\ind_{\B \cap \K \Z}^{\B} V$ l'induite à support comapct et on note $[b,v]$ (comme dans \cite{BL2} et \cite{BR1}) l'élément de $\ind_{\B \cap \K \Z}^{\B} V$ défini par $[b,v](g) = gb \cdot v$ si $gb \in \B \cap \K \Z$ et $[b,v](g) = 0$ sinon. On déduit du lemme ci-dessus un morphisme $\B \cap \K \Z$-équivariant de la représentation $(\omega^r \otimes 1) \otimes (\chi \circ \det)$ vers $\Omega(W)$ et par réciprocité de Frobenius, on en déduit l'existence d'un morphisme : 
\[ \pi_W : \ind_{\B \cap \K \Z}^{\B} (\omega^r \otimes 1) \otimes (\chi \circ \det) 
\to \Omega(W), \]
déterminé par $\pi_W([1,v_r]) = \theta$ où $v_r$ est une base de $(\omega^r \otimes 1) \otimes (\chi \circ \det)$.

\begin{prop}\label{heckesurnul}
Soit $J = \{ j_0,\hdots,j_{p-1}\}$ un ensemble d'éléments de $\Zp$ tels que $j_i = i \mod{p}$ pour tout $i$.
\begin{enumerate}
\item Si $r=0$, alors $\pi_W(\smat{1 & 0\\0 & p}  [1,v_0] + \sum_{j \in J} \smat{p & j \\ 0 & 1} [1,v_0])  = 0$; 
\item si $r \geq 1$, alors $\pi_W(\sum_{j \in J} (-j)^k \smat{p & j \\ 0 & 1} [1,v_r])=0$ pour tout $k \in \{ 0,\hdots,r-1\}$;
\item si $r \geq 1$ et si $\lambda_0,\hdots,\lambda_{p-1}$ sont tels que $\sum_{i=0}^{p-1} i^{\ell} \lambda_i = 0$ pour tout $0 \leq \ell \leq r-1$, alors $\pi_W( \sum_{i=0}^{p-1} \lambda_i i^r [1,v_r] + \sum_{i=0}^{p-1} \lambda_i \smat{1 & i \\ 0 & 1} \smat{1 & 0 \\ 0 & p^{-1}} \sum_{j \in J} (-j)^r \smat{p & j \\ 0 & 1} [1,v_r])=0$.
\end{enumerate}
\end{prop}

\begin{proof}
Pour montrer le (1), il faut vérifier que :
\[ \theta\left(\pmat{1 & 0\\0 & p^{-1}} \star y + \sum_{j \in J} \pmat{p^{-1} & -j p^{-1} \\ 0 & 1} \star y\right)=0 \] 
quel que soit $y \in \projlim_\psi \ddiese(W)$. En utilisant le fait que :
\[ \pmat{p^{-1} & -j p^{-1} \\ 0 & 1} = \pmat{p^{-1} & 0 \\ 0 & p^{-1}} \pmat{1 & 0 \\ 0 & p} \pmat{1 & -j \\ 0 & 1} \] 
et les formules donnant l'action de $\B$ sur $y \in \projlim_\psi \ddiese(W)$, on se ramène à montrer que :
\[ \theta\left(\psi^{-1}(y) + \lambda^2 \sum_{j \in J} \psi( (1+X)^{-j} y)\right)=0. \]
Le fait que $\theta((1+X)^k z)=\theta(z)$ si $k \in \Zp$ et $z \in \projlim_\psi \ddiese(W)$ montre que la valeur de $\theta(\psi^{-1}(y) + \lambda^2 \sum_{j \in J} \psi( (1+X)^{-j} y))$ est inchangée si l'on remplace $J$ par un autre système de représentants de $\Fp$ dans $\Zp$ et on choisit $J=\{0,-1,\hdots,-(p-1)\}$. Comme on a alors $\sum_{j \in J} (1+X)^{-j}=X^{p-1}$, on s'est ramené à montrer que $\theta(\psi^{-1}(y) + \lambda^2 \psi( X^{p-1} y))=0$. En écrivant $y_1=\alpha_1 e + X^r \beta_1 f$ et en posant $y=\psi(\psi^{-1}(y))$, on trouve que $\theta(\psi^{-1}(y) + \lambda^2 \psi( X^{p-1} y))= \alpha_1(0) - \psi(X^{p-1} \psi (X^{p-1} \alpha_1))(0)$ et c'est un petit exercice de vérifier que $\psi(X^{p-1} \psi (X^{p-1} \alpha_1))(0) = \alpha_1(0)$.

Pour montrer le (2), il faut vérifier que :
\[ \theta\left( \sum_{j \in J} (-j)^k \pmat{p^{-1} & -j p^{-1} \\ 0 & 1} \star y\right)=0 \] 
pour tout $k \in \{ 0,\hdots,r-1\}$, et comme pour le cas $r=0$, on se ramène à montrer que $\theta(\psi(\sum_{j \in J} (-j)^k (1+X)^{-j} y)) = 0$. Comme ci-dessus, la valeur de cette expression ne dépend du choix de $J$ et on prend $J=\{0,-1,\hdots,-(p-1)\}$. En écrivant $y_0 = \alpha_0 e+ X^r \beta_0 f$, on se ramène à montrer que la série $\psi(X^r \beta_0 \sum_{j =0}^{p-1} j^k (1+X)^j)$ est nulle en $0$. Le coefficient de $X^t$ dans $\sum_{j=0}^{p-1} j^k (1+X)^j$ est $ \sum_{j=0}^{p-1} j^k \binom{j}{t}$ et cette somme est nulle (dans $\Fp$) tant que $k+t \leq p-1$ (encore un exercice) ce qui fait que $X^r \beta_0 \sum_{j =0}^{p-1} j^k (1+X)^j$ est divisible par $X^p$ si $k \leq r-1$.

Pour montrer le (3), il faut vérifier que :
\[ \theta \left( \sum_{i=0}^{p-1} \lambda_i i^r y + \sum_{j \in J} \sum_{i=0}^{p-1} (-j)^r  \lambda_i \pmat{p^{-1} & -j p^{-1} \\ 0 & 1} \pmat{1 & 0 \\ 0 & p} \pmat{1 & -i \\ 0 & 1} \star y \right) =0, \]
et comme ci-dessus, on se ramène à montrer que $\theta(z)=0$ où :
\[ z  = \sum_{i=0}^{p-1} \lambda_i i^r y + \lambda^2 \psi\left(\sum_{j \in J} (-j)^r (1+X)^{-j}  \psi\left(\sum_{i=0}^{p-1} \lambda_i (1+X)^{-i} y\right) \right). \]
Le premier terme non nul de la série $\sum_{j \in J} (-j)^r (1+X)^{-j}$ est $-X^{p-1-r} / (p-1-r)!$ et le fait que l'on a $\sum_{i=0}^{p-1} i^{\ell} \lambda_i = 0$ pour tout $0 \leq \ell \leq r-1$ implique que tous les termes de la série $\sum_{i=0}^{p-1} \lambda_i (1+X)^{-i}$ sont nuls jusqu'à $X^r$ dont le coefficient vaut $(-1)^r (\sum_{i=0}^{p-1} i^r \lambda_i) / r!$. Le théorème de Wilson implique que $r! \cdot (p-1-r)! = (-1)^{r+1}$ dans $\Fp$ et un calcul semblable à celui du (1) montre que si l'on écrit $y=\alpha e + \beta X^r f$, alors le coefficient de $e$ dans $z$ est $\sum_{i=0}^{p-1} \lambda_i i^r (\alpha - \psi((X^{p-1} + O(X^p)) \psi((X^{p-1} + O(X^p)) \alpha)))$ qui est bien nul en $X=0$.
\end{proof}

\section{Démonstration de l'isomorphisme}

Ce {\S} est consacré à la démonstration du théorème A de l'introduction. Rappelons que $\Sym^r k^2$ est l'ensemble des polynômes homogènes en $x$ et $y$ de degré $r$ à coefficients dans $k$ muni de l'action de $\K$ donnée par $\smat{a & b \\ c & d} P(x,y) = P(ax+cy,bx+dy)$ et qu'on étend l'action de $\K$ à $\K\Z$ en décidant que $\smat{p & 0 \\ 0 & p} P(x,y) = P(x,y)$. Afin de montrer le théorème A de l'introduction, il reste à faire le lien entre les représentations $\ind_{\B \cap \K \Z}^{\B} (\omega^r \otimes 1)$ et $\ind_{\K \Z}^{\G} \Sym^r k^2$. Rappelons à cet effet que l'on dispose de la décomposition d'Iwasawa $\G=\B\K$ qui a pour conséquence que si $V$ est une représentation de $\K\Z$, alors l'application {\og restriction à $\B$ \fg} de $\ind_{\K\Z}^{\G} V$ vers $\ind_{\B \cap \K\Z}^{\B} V$ est un isomorphisme. 

Par ailleurs, la représentation $\omega^r \otimes 1$ est une sous-$\B$-représentation de $\Sym^r k^2$ (elle se réalise sur l'espace engendré par $x^r$) et on en déduit une application $\ind_{\B \cap \K\Z}^{\B} \omega^r \otimes 1 \to \ind_{\B \cap \K\Z}^{\B} \Sym^r k^2$. Rappelons que $T$ désigne l'opérateur de Hecke défini dans \cite{BL} et \cite{BL2}.

\begin{prop}\label{isombor}
La représentation $(\ind_{\B \cap \K\Z}^{\B} 1) / T$ est irréductible et si $r \geq 1$, alors l'application :
\[ \frac{\ind_{\B \cap \K\Z}^{\B} (\omega^r \otimes 1)}{T (\ind_{\B \cap \K\Z}^{\B} \Sym^r k^2) \cap \ind_{\B \cap \K\Z}^{\B} (\omega^r \otimes 1)}   \to \frac{\ind_{\B \cap \K\Z}^{\B} \Sym^r k^2}{T(\ind_{\B \cap \K\Z}^{\B} \Sym^r k^2)} \]
est un isomorphisme et les deux représentations sont irréductibles.
\end{prop}

\begin{proof}
Le fait que $(\ind_{\B \cap \K\Z}^{\B} 1) / T$ et $(\ind_{\B \cap \K\Z}^{\B} \Sym^r k^2) / T$ sont irréductibles fait l'objet du (i) du théorème 1.1 de \cite{P06} étant donné l'isomorphisme entre $\ind_{\K\Z}^{\G} \Sym^r k^2$ et $\ind_{\B \cap \K\Z}^{\B} \Sym^r k^2$ rappelé ci-dessus. Quand $r \geq 1$,  l'application donnée est injective par construction et son image est une sous-représentation non-triviale de $(\ind_{\B \cap \K\Z}^{\B} \Sym^r k^2) / T$. Cette représentation étant irréductible, l'application ci-dessus est bien un isomorphisme.
\end{proof}

Rappelons que des formules donnant l'action de $T$ sur $\ind_{\B \cap \K\Z}^{\B} \Sym^r k^2$ se trouvent dans le \S 2.2 de \cite{BR2}. On a notamment :
\begin{equation} \label{formule}
T ( [1,x^{r-i}y^i] ) =  
\begin{cases} \sum_{j \in J}  \smat{p & j \\ 0 & 1} [1, (-j)^i x^r] & \text{si $i \leq r-1$;} \\
\smat{1 & 0 \\ 0 & p} [1,y^r] + \sum_{j \in J}  \smat{p & j \\ 0 & 1} [1, (-j)^r x^r] & \text{si $i = r$;}
\end{cases} 
\end{equation}

\begin{lemm}\label{noyr}
Si $r \geq 1$, alors $T (\ind_{\B \cap \K\Z}^{\B} \Sym^r k^2) \cap \ind_{\B \cap \K\Z}^{\B} (\omega^r \otimes 1)$ est engendré par les translatés sous l'action de $\B$ des vecteurs :
\[ \begin{cases}
T([1,x^{r-i}y^i]) & \text{pour $0 \leq i \leq r-1$,} \\
T( \sum_{i=0}^{p-1} \lambda_i [ \smat{1 & p^{-1}i \\ 0 & p^{-1}} , y^r ]) & \text{où $\sum_{i=0}^{p-1} i^{\ell} \lambda_i = 0$ pour tout $0 \leq \ell \leq r-1$}.
\end{cases} \]
\end{lemm}

\begin{proof}
La formule (\ref{formule}) ci-dessus implique que $T([1,x^{r-i}y^i]) \in \ind_{\B \cap \K\Z}^{\B} (\omega^r \otimes 1)$ et donc de même pour les translatés de ces vecteurs. Il reste donc à déterminer quand est-ce qu'un vecteur du type $T(\sum_j [b_j, \lambda_j y^r])$ appartient à $\ind_{\B \cap \K\Z}^{\B} (\omega^r \otimes 1)$. Pour cela, soit $A=\{ \alpha_n p^{-n} + \cdots + \alpha_1 p^{-1}$ où $0 \leq \alpha_j \leq p-1 \}$ ce qui fait que $A$ est un système de représentants de $\Qp / \Zp$ et que l'on a :
\[ \B = \coprod_{\substack{\beta \in A \\ \delta\in \ZZ}} \pmat{ 1 & \beta \\ 0 & p^\delta} \B \cap \K\Z, \]
comme le montre un petit calcul. On pose $b_{\beta,\delta} = \smat{ 1 & \beta \\ 0 & p^\delta}$ et on se donne un vecteur $v$ de la forme $\sum_{\beta,\delta} \sum_{i=0}^{p-1} \lambda_{\beta,\delta,i} [b_{p^{-1} \beta + p^{-1} i ,\delta}, y^r]$ (remarquons que $A = \coprod_{i=0}^{p-1} p^{-1} A + p^{-1} i$). 
On a alors :
\[ T(v) = \sum_{\beta,\delta} b_{\beta,\delta+1} \cdot T \left( \lambda_{\beta,\delta,0} [ \smat{1 & 0 \cdot p^{-1} \\ 0 & p^{-1}},y^r] + \cdots + \lambda_{\beta,\delta,p-1} [ \smat{1 & (p-1) \cdot p^{-1} \\ 0 & p^{-1}},y^r] \right), \]
ce qui fait que l'ensemble des vecteurs $v$ tels que $T(v) \in \ind_{\B \cap \K\Z}^{\B} (\omega^r \otimes 1)$ est engendré par les translatés sous l'action de $\B$ des $v_\lambda =  \sum_{i=0}^{p-1} \lambda_i [ \smat{1 & p^{-1}i \\ 0 & p^{-1}} , y^r ]$ tels que $T(v_\lambda) \in \ind_{\B \cap \K\Z}^{\B} (\omega^r \otimes 1)$. La formule (\ref{formule}) montre que c'est le cas si et seulement si $\sum_{i=0}^{p-1} \lambda_i [ \smat{1 & i \\ 0 & 1} , y^r ] \in \ind_{\B \cap \K\Z}^{\B} (\omega^r \otimes 1)$ et donc si et seulement si $\sum_{i=0}^{p-1} \lambda_i (ix+y)^r \in k \cdot x^r$ ce qui est équivalent aux conditions $\sum_{i=0}^{p-1} i^{\ell} \lambda_i = 0$ pour tout $0 \leq \ell \leq r-1$.
\end{proof}

\begin{prop}\label{enghecke}
Soit $J = \{ j_0,\hdots,j_{p-1}\}$ un ensemble d'éléments de $\Zp$ tels que $j_i = i \mod{p}$ pour tout $i$. Si $r=0$, alors $T(\ind_{\B \cap \K\Z}^{\B} 1)$ est engendré par les translatés de :
\[ \pmat{1 & 0\\0 & p}  [1,1] + \sum_{j \in J} \pmat{p & j \\ 0 & 1} [1,1]  \] et si $r \geq 1$, alors $T (\ind_{\B \cap \K\Z}^{\B} \Sym^r k^2) \cap \ind_{\B \cap \K\Z}^{\B} (\omega^r \otimes 1)$ est engendré par les translatés sous l'action de $\B$ des  :
\[ \sum_{j \in J} (-j)^i \pmat{p & j \\ 0 & 1} [1,x^r],  \]
pour $0 \leq i \leq r-1$ et des :
\[ \sum_{i=0}^{p-1} \lambda_i i^r [1,x^r] + \sum_{i=0}^{p-1} \lambda_i \pmat{1 & p^{-1} i \\ 0 & p^{-1}} \sum_{j \in J} (-j)^r \pmat{p & j \\ 0 & 1} [1,x^r], \]
où $\lambda_0,\hdots,\lambda_{p-1}$ sont tels que $\sum_{i=0}^{p-1} i^{\ell} \lambda_i = 0$ pour tout $0 \leq \ell \leq r-1$.
\end{prop}

\begin{proof}
Comme $\ind_{\B \cap \K\Z}^{\B} 1$ est engendrée par les translatés de $[1,1]$ sous l'action de $\B$, le sous-espace $T(\ind_{\B \cap \K\Z}^{\B} 1)$ est engendré par les translatés de $T([1,1]) = \smat{1 & 0\\0 & p}  [1,1] + \sum_{j \in J} \smat{p & j \\ 0 & 1} [1,1]$ ce qui montre le premier point.

Si $r \geq 1$, alors la formule (1) nous dit que $T([1,x^{r-i}y^i]) =  \sum_{j \in J}  \smat{p & j \\ 0 & 1} [1, (-j)^i x^r]$ pour $i \leq r-1$ et que :
\[ T\left( \sum_{i=0}^{p-1} \lambda_i [ \smat{1 & p^{-1}i \\ 0 & p^{-1}} , y^r ]\right) = \sum_{i=0}^{p-1} \lambda_i [\smat{1 & i \\ 0 & 1},y^r] + \sum_{i=0}^{p-1} \lambda_i \smat{1 & p^{-1} i \\ 0 & p^{-1}} \sum_{j \in J} (-j)^r \smat{p & j \\ 0 & 1} [1,x^r]. \]
La condition $\sum_{i=0}^{p-1} i^{\ell} \lambda_i = 0$ pour tout $0 \leq \ell \leq r-1$ implique que $\sum_{i=0}^{p-1} \lambda_i [\smat{1 & i \\ 0 & 1},y^r] = \sum_{i=0}^{p-1} \lambda_i i^r [1,x^r]$ et le lemme \ref{noyr} permet alors de conclure.
\end{proof}

\begin{proof}[Démonstration du théorème A]
Rappelons que l'on a construit au \S 2 une application $\pi_W : \ind_{\B \cap \K \Z}^{\B} (\omega^r \otimes 1) \otimes (\chi \circ \det) \to \Omega(W)$, déterminée par $\pi_W([1,x^r]) = \theta$ (où l'on identifie $(\omega^r \otimes 1) \otimes (\chi \circ \det)$ à la sous-représentation de $\Sym^r k^2  \otimes (\chi \circ \det)$ engendrée par $x^r$). Etant donné que la représentation $\Omega(W)$ est irréductible, et en vertu des isomorphismes entre représentations irréductibles fournis par la proposition \ref{isombor}, il ne reste plus qu'à montrer que l'image par $\pi_W$ de $T(\ind_{\B \cap \K\Z}^{\B} 1) \otimes (\chi \circ \det)$ (si $r=0$) ou de 
$T (\ind_{\B \cap \K\Z}^{\B} \Sym^r k^2) \otimes (\chi \circ \det) \cap \ind_{\B \cap \K\Z}^{\B} (\omega^r \otimes 1) \otimes (\chi \circ \det)$ (si $r \geq 1$) est nulle, ce qui suit directement de la proposition \ref{enghecke} et de la proposition \ref{heckesurnul}.
\end{proof}

\vspace{20pt}\noindent \textbf{Remerciements} : je remercie Christophe Breuil de m'avoir demandé (en 2005!) une démons\-tration directe de l'isomorphisme du théorème A ainsi que pour les nombreuses discussions que nous avons eues sur {\og Langlands $p$-adique \fg}.

\end{document}